\documentclass{article}
\pdfoutput=1  

\usepackage{graphicx,hyperref} 
\hypersetup{debug=true} 

\usepackage[english]{babel}
\usepackage{amsmath,amsfonts}
\usepackage[psamsfonts]{amssymb} 
\usepackage{mdwlist,tensor}
\usepackage{enumerate,url,enumitem,multirow,multicol}
\usepackage{algorithmic}
\usepackage[ruled,vlined]{algorithm2e} 

\usepackage{booktabs}
\usepackage{amsthm}
\usepackage{caption,subcaption}

\newtheorem{theorem}{Theorem}

\newtheorem{lemma}[theorem]{Lemma}
\newtheorem{proposition}[theorem]{Proposition}
\newtheorem{corollary}[theorem]{Corollary}

\newtheorem{definition}[theorem]{Definition}
\newtheorem{example}[theorem]{Example}
\newtheorem{remark}[theorem]{Remark}

\newcommand{\vx}{{x}}

\newcommand{\RR}{\mathbb{R}}

\newcommand{\h}{p}
\newcommand{\hh}{\hat{p}} 
\newcommand{\f}{h}

\newcommand{\be}{\begin{eqnarray}}
\newcommand{\ee}{\end{eqnarray}}

\renewcommand{\P}{\mathsf{P}}

\newcommand{\C}{\mathcal{C}}
\renewcommand{\int}{\mathrm{int}}

\newcommand{\norm}[1]{\left\|#1\right\|}
\newcommand{\pds}[2]{\left\langle#1,#2\right\rangle}

\newcommand{\parenth}[1]{\left({#1}\right)}

\newcommand{\acc}[1]{\left\{{#1}\right\}}

\newcommand{\Hm}{\mathcal{H}}
\newcommand{\Km}{\mathcal{K}}

\newcommand{\beqn}{\begin{eqnarray}}
\newcommand{\eeqn}{\end{eqnarray}}

\renewcommand{\ge}{\geqslant}
\renewcommand{\leq}{\leqslant}

\newcommand{\Id}{\mathrm{I}}

\renewcommand{\iff}{\Leftrightarrow}
\newcommand{\eqdef}{\triangleq}
\newcommand{\sdp}{\mathbb{S}_{++}}

\newcommand{\eg}{\textit{e.g.}~}
\newcommand{\ie}{\textit{i.e.}~}
\newcommand{\etc}{\textit{etc.}\xspace}

\DeclareMathOperator{\prox}{prox}

\newcommand{\infc}{\stackrel{\mathrm{+}}{\vee}}
\newcommand{\indic}{\imath}
\newcommand{\env}[2]{\tensor*[^#2]{#1}{}}
\DeclareMathOperator{\dom}{dom}
\DeclareMathOperator{\range}{Im}

\DeclareMathOperator{\ri}{ri}

\DeclareMathOperator*{\argmin}{argmin}
\DeclareMathOperator*{\Argmin}{Argmin}

\DeclareMathAlphabet{\mathbbb}{U}{bbold}{m}{n}
\newcommand{\ones}{\mathbbb 1}                

\newcommand{\BBa}{\tau_{\text{BB}2}}
\newcommand{\BBb}{\tau_{\text{BB}1}}

\title{A quasi-Newton proximal splitting method}

\author{S. Becker\thanks{LJLL, CNRS-UPMC, Paris France (\tt{stephen.becker@upmc.fr}).}
\and
M.J. Fadili\thanks{GREYC, CNRS-ENSICAEN-Universit\'e de Caen, 14050 Caen France (\tt{Jalal.Fadili@greyc.ensicaen.fr}).} 
}
\date{\today}
\begin{document}

\maketitle

\begin{abstract}
A new result in convex analysis on the calculation of proximity
operators in certain scaled norms is derived.
We describe efficient implementations
of the proximity calculation for a useful class of functions; the implementations
exploit the piece-wise linear nature of the dual problem.
The second part of the paper applies the previous result to acceleration of convex
minimization problems, and leads to an elegant quasi-Newton method.
The optimization method compares favorably against state-of-the-art
alternatives. The algorithm has extensive applications including signal processing, sparse recovery and machine learning and classification.
\end{abstract}

\section{Introduction}
\label{sec:intro}

Convex optimization has proved to be extremely useful to all quantitative disciplines  
of science. A common trend in modern science is the increase in size of datasets,
which drives the need for more efficient optimization schemes. For large-scale 
unconstrained smooth convex
problems, two classes of methods have seen the most success: limited memory quasi-Newton methods
and non-linear conjugate gradient (CG) methods. Both of these methods generally outperform
simpler methods, such as gradient descent.

For problems with non-smooth terms and/or constraints, it is possible to generalize
gradient descent with \emph{proximal gradient descent} (which includes projected gradient
descent as a sub-cases), which is just the application of the forward-backward algorithm~\cite{BauschkeCombettes11}.

Unlike gradient descent, it is not easy to adapt quasi-Newton and CG methods to problems involving
constraints and non-smooth terms.  Much work has been written on the topic, and approaches generally follow an active-set methodology. In the limit, as the active-set is correctly identified, the methods
behave similar to their unconstrained counterparts. These methods have seen success, but are not
as efficient or as elegant as the unconstrained versions. In particular, a sub-problem on the active-set
must be solved, and the accuracy of this sub-iteration must be tuned with heuristics in order
to obtain competitive results.

\subsection{Problem statement}
\label{sec:statement}
Let $\Hm=(\RR^N,\pds{\cdot}{\cdot})$ equipped with the usual Euclidean scalar product $\pds{x}{y}=\sum_{i=1}^N x_iy_i$ and associated norm $\norm{x}=\sqrt{\pds{x}{x}}$. For a matrix $V \in \RR^{N \times N}$ in the symmetric positive-definite (SDP) cone $\sdp(N)$, we define $\Hm_V=(\RR^N,\pds{\cdot}{\cdot}_V)$ with the scalar product $\pds{x}{y}_V = \pds{x}{Vy}$ and norm $\norm{x}_V$ corresponding to the metric induced by $V$. The dual space of $\Hm_V$, under $\pds{\cdot}{\cdot}$, is $\Hm_{V^{-1}}$. We denote $\Id_\Hm$ the identity operator on $\Hm$.

A real-valued function $f: \Hm \to \RR \cup \acc{+\infty}$ is (0)-\emph{coercive} if $\lim_{\norm{\vx} \to +\infty}f\parenth{\vx}=+\infty$. The \emph{domain} of $f$ is defined by $\dom f = \{ x\in\Hm\ :\ f(x) < +\infty \}$ and $f$ is \emph{proper} if $\dom f \neq
\emptyset$. We say that a real-valued function $f$ is \emph{lower semi-continuous} (lsc) if $\liminf_{x \to x_0} f(x) \geq f(x_0)$. The class of all proper lsc convex functions from $\Hm$ to $\RR \cup \acc{+\infty}$ is denoted by $\Gamma_0(\Hm)$. The conjugate or Legendre-Fenchel transform of $f$ on $\Hm$ is denoted $f^*$ .

Our goal is the generic minimization of functions of the form 
\begin{equation}\tag{$\P$}
\label{eq:minP}
\min_{x \in \Hm} ~ \{F(x) \eqdef f(x) + h(x)\}~,
\end{equation}
where $f,h \in \Gamma_0(\Hm)$. 
We also assume 
the set of minimizers is nonempty (e.g.~$F$ is coercive) and that a standard
domain qualification holds. We take $f \in C^1(\RR^N)$ with $L$-Lipschitz continuous gradient,
and we assume $h$ is separable. 
Write $x^\star$ to denote an element of $\Argmin F(x)$.

The class we consider covers non-smooth convex optimization problems, including those with convex constraints.
Here are some examples in regression, machine learning and classification.
\begin{example}[LASSO]
    \begin{equation}
        \min_{x \in \Hm} \frac{1}{2}\|Ax-b\|_2^2 + \lambda \|x\|_1 ~.
        \label{eq:LASSO}
    \end{equation}
\end{example}
\begin{example}[Non-negative least-squares (NNLS)]
        \begin{equation}
            \min_{x \in \Hm} \frac{1}{2}\|Ax-b\|_2^2 \quad\text{subject to}\quad x \ge 0 ~.
        \label{eq:NNLS}
    \end{equation}
\end{example}
\begin{example}[Sparse Support Vector Machines]
One would like to find a linear decision function which minimizes the objective
    \begin{equation}
        \min_{x \in \Hm} \frac{1}{m}\sum_{i=1}^m L(\pds{x}{z_i}+b,y_i) + \lambda \|x\|_1
        \label{eq:hingesvm}
    \end{equation}
where for $i=1,\cdots,m$, $(z_i,y_i) \in \RR^N \times \{\pm 1\}$ is the training set, and $L$ 
is a smooth loss function with Lipschitz-continuous gradient such as
the squared hinge loss $L(\hat{y}_i,y_i)=\max(0,1-\hat{y}_iy_i)^2$ 
or the logistic loss $L(\hat{y}_i,y_i)=\log(1+e^{-\hat{y}_iy_i})$.
\end{example}

\subsection{Contributions} 
This paper introduces a class of scaled norms for which we can compute a proximity operator; these results themselves are significant, for previous results only cover diagonal scaling (the diagonal scaling result is trivial). Then, motivated by the discrepancy between constrained and unconstrained performance, we define a class of limited-memory quasi-Newton methods to solve \eqref{eq:minP} and that extends naturally and elegantly from the unconstrained to the constrained case. Most well-known quasi-Newton methods for constrained problems, such as L-BFGS-B~\cite{LBFGSB}, are only applicable to box constraints $ l \le x \le u$.  The power of our approach is that it applies to a wide-variety of useful non-smooth functionals (see \S\ref{sec:examples} for a list) and that it does not rely on an active-set strategy. The approach uses the zero-memory SR1 algorithm, and we provide evidence that the non-diagonal term provides significant improvements over diagonal Hessians.

\section{Quasi-Newton forward-backward splitting}
\subsection{The algorithm}
In the following, define the quadratic approximation
\begin{equation}\label{eq:Q}
    Q_k^B(x) = f(x_k) + \pds{\nabla f(x_k)}{ x-x_k} + \frac{1}{2}\|x-x_k\|_{B}^2,  
\end{equation}
where $B \in \sdp(N)$.

The standard (non relaxed) version of the forward-backward splitting algorithm (also known as proximal or projected gradient descent) to solve \eqref{eq:minP} updates to a new iterate $x_{k+1}$ according to
\begin{equation} \label{eq:firstOrder}
    x_{k+1} = \argmin_x Q_k^{B_k}(x) + h(x) = \prox_{t_kh}( x_k - t_k\nabla f(x_k) )
\end{equation}
with $B_k = t_k^{-1} \Id_{\Hm}$, $t_k \in ]0,2/L[$ (typically $t_k=1/L$ unless a line search is used).

Note that this specializes to the gradient descent when $h=0$. Therefore, if $f$ is a strictly convex quadratic function and one takes $B_k=\nabla^2 f(x_k)$, then we obtain the Newton method. Let's get back to $h \neq 0$. It is now well known that fixed $B=L\Id_{\Hm}$ is usually a poor choice. Since $f$ is smooth and can be approximated by a quadratic, and inspired by quasi-Newton methods, this suggest picking $B_k$ as an approximation of the Hessian. Here we propose a diagonal+rank 1 approximation.  

Our diagonal+rank 1 quasi-Newton forward-backward splitting algorithm is listed in Algorithm~\ref{alg:main} (with details for the quasi-Newton update in Algorithm~\ref{alg:SR1}, see \S\ref{sec:SR1} for details). These algorithms are listed as simply as possible to emphasize their important components; the actual software used for numerical tests is open-source and available at \url{http://www.greyc.ensicaen.fr/~jfadili/software.html}. 

\begin{algorithm}[H]
    \caption{Zero-memory Symmetric Rank 1 (0SR1) algorithm to solve $\min f+h$ \label{alg:main} }
\begin{algorithmic}[1]
\REQUIRE $x_0\in\dom(f+h)$, Lipschitz constant estimate $L$ of $\nabla f$, stopping criterion $\epsilon$
  \FOR{$k=1,2,3,\dots$}
  \STATE $s_k \leftarrow x_k - x_{k-1}$
  \STATE $y_k \leftarrow \nabla f(x_k) - \nabla f(x_{k-1})$
  \STATE Compute $H_k$ via Algorithm~\ref{alg:SR1}, and define $B_k = H_k^{-1}$.
  \STATE Compute the rank-1 proximity operator (see \S\ref{sec:prox})
  \begin{equation}
      \hat{x}_{k+1} \leftarrow  \prox^{B_k}_{h}( x_k - H_k \nabla f(x_k) )  
  \end{equation}
  \STATE $p_k \leftarrow \hat{x}_{k+1} - x_k$ and 
  terminate if $\|p_k\| < \epsilon $
  \STATE Line-search along the ray $x_k + t p_k$ to determine $x_{k+1}$, or choose $t=1$.
  \ENDFOR
\end{algorithmic}  
\end{algorithm}

\subsection{Relation to prior work}
\label{sec:relation-prior-work}

\paragraph{First-order methods}
The algorithm in \eqref{eq:firstOrder} is variously known as proximal descent or iterated shrinkage/thresholding algorithm (IST or ISTA). It has a grounded convergence theory, and also admits over-relaxation factors $\alpha \in (0,1)$~\cite{CombettesPesquetChapter}.

 The spectral projected gradient (SPG)~\cite{SPG} method was designed as an extension of the Barzilai-Borwein spectral step-length method to constrained problems. 
 In~\cite{WrightSparsa08}, it was extended to non-smooth problems by allowing general proximity operators; we refer to this as SPG/SpaRSA (N.B.~we do not use the SpaRSA implementation since we do not use warm-starts or restarts, in order to be fair to all algorithms).
 The Barzilai-Borwein method~\cite{BB88} use a specific choice of step-length $t_k$ motivated by quasi-Newton methods. Numerical evidence suggests the SPG/SpaRSA method is highly effective, although convergence results are not as strong as for ISTA. 

        FISTA~\cite{FISTA} is a multi-step accelerated version of ISTA inspired by the work of Nesterov. The stepsize $t$ is chosen in a similar way to ISTA; in our implementation, we tweak the original approach by using a Barzilai-Borwein step size, a standard line search, and restart\cite{restart}, since this led to improved performance.
        Nesterov acceleration can be viewed as an over-relaxed version of ISTA with a specific, non-constant over-relaxation parameter $\alpha_k$.

The above approaches assume $B_k$ is a constant diagonal.
The general diagonal case was considered in several papers in the 1980s as a simple quasi-Newton method,
but never widely adapted. More recent attempts include a static choice $B_k \equiv B$ for a primal-dual method~\cite{ChambollePock11b}. A convergence rate analysis of forward-backward splitting with static and variable $B_k$ where one of the operators is maximal strongly monotone is given in \cite{ChenRockafellar97}.

\paragraph{Active set approaches} Active set methods take a simple step, such as gradient projection, to identify active variables, and then uses a more advanced quadratic model to solve for the free variables. A well-known such method is L-BFGS-B \cite{LBFGSB,L-BFGS-B-97} which handles general box-constrained problems; we test an updated version~\cite{LBFGSB2011}. 
A recent bound-constrained solver is ASA~\cite{HagerZhang06} which uses a conjugate gradient (CG) solver on the free variables, and shows good results compared to L-BFGS-B, SPG, GENCAN and TRON.
We also compare to several active set approaches specialized for $\ell_1$ penalties:
``Orthant-wise Learning'' (OWL)~\cite{AndrewGao07}, 
``Projected Scaled Sub-gradient + Active Set'' (PSSas)~\cite{Schmidt2007a},
``Fixed-point continuation + Active Set'' (FPC\_AS)~\cite{FPCAS},
and ``CG + IST'' (CGIST)~\cite{GoldsteinSetzer11}.

\paragraph{Other approaches}
By transforming the problem into a standard conic programming problem, the generic problem is amenable to interior-point methods (IPM). IPM requires solving a Newton-step equation, so first-order like ``Hessian-free'' variants of IPM solve the Newton-step approximately, either by approximately solving the equation or by subsampling the Hessian. The main issues are speed and robust stopping criteria for the approximations.  

Yet another approach is to include the non-smooth $h$ term in the quadratic approximation. Yu et al.~\cite{YuVishwanathan10} propose a non-smooth modification of BFGS and L-BFGS, and test on problems where $h$ is typically a hinge-loss or related function. 

 The projected quasi-Newton (PQN) algorithm~\cite{projQuasiNewton09,projNewton11} is perhaps the most elegant
and logical extension of quasi-Newton methods, but it involves solving a sub-iteration. 
PQN proposes the SPG~\cite{SPG}
algorithm for the subproblems, and finds that this is an efficient tradeoff whenever the cost function (which is not
involved in the sub-iteration) is relatively much more expensive to evaluate than projecting onto the constraints.
Again, the cost of the sub-problem solver (and a suitable stopping criteria for this inner solve) are issues. 
As discussed in~\cite{Saunders12},
it is possible to generalize PQN to general non-smooth problems whenever the proximity operator
is known (since, as mentioned above, it is possible to extend SPG to this case).

\section{Proximity operators and proximal calculus}
\label{sec:prox}
We only recall essential definitions. More notions and results from convex analysis can be found in \S\ref{sec:appendix}.

\begin{definition}[Proximity operator \cite{Moreau1962}]
\label{def:prox} 
Let $\f \in \Gamma_0(\Hm)$. Then, for every $x\in\Hm$, the
function $z \mapsto \frac{1}{2}\norm{x-z}^{2} + \f(z)$ achieves its infimum at a unique
point denoted by $\prox_{\f}x$. The uniquely-valued operator $\prox_{\f}: \Hm \to \Hm$ thus
defined is the \textit{proximity operator} or proximal mapping of $\f$.
\end{definition}

\subsection{Proximal calculus in \texorpdfstring{$\Hm_V$}{Hv}}
Throughout, we denote $\prox^V_\f=(\Id_{\Hm_V}+V^{-1} \partial \f)^{-1}$, where $\partial \f$ is the subdifferential of $\f$, the proximity operator of $\f$ w.r.t. the norm endowing $\Hm_V$ for some $V \in \sdp(N)$. Note that since $V \in \sdp(N)$, the proximity operator $\prox^V_\f$ is well-defined. 

\begin{lemma}[Moreau identity in $\Hm_V$]
\label{lem:moreauidentV}
Let $\f \in \Gamma_0(\Hm)$, then for any $x \in \Hm$
\be
\prox^V_{\rho \f^*}(x) + \rho V^{-1} \circ \prox^{V^{-1}}_{\f/\rho} \circ V(x/\rho) = x, \forall ~ 0 < \rho < +\infty ~.
\ee
\end{lemma}
The proof is in \S\ref{appendix:lem:moreauidentV}.

\begin{corollary}
\label{cor:proxV}
\be
\prox^V_{\f}(x) = x - V^{-1} \circ \prox^{V^{-1}}_{\f^*} \circ V(x) ~.
\ee
\end{corollary}

\subsubsection{Diagonal+rank-1: General case}
\begin{theorem}[Proximity operator in $\Hm_V$]
\label{theo:proxVrank1}
Let $\f \in \Gamma_0(\Hm)$ and $V=D+uu^T$, where $D$ is diagonal with (strictly) positive diagonal elements $d_i$, and $u \in \RR^N$. Then,
\be
\label{eq:proxVgen}
\prox^V_\f(x) = D^{-1/2} \circ \prox_{\f \circ D^{-1/2}}(D^{1/2}x - v) ~,
\ee
where $v = \alpha D^{-1/2}u$ and $\alpha$ is the unique root of
\be  
\label{eq:roothgen}
\h(\alpha) = \pds{u}{x - D^{-1/2} \circ \prox_{\f \circ D^{-1/2}}\circ D^{1/2}(x - \alpha D^{-1} u)} + \alpha ~,
\ee
which is a Lipschitz continuous and strictly increasing function on $\RR$ with Lipschitz constant $1+\sum_i u_i^2/d_i$.
\end{theorem}
The proof is in \S\ref{appendix:thm:proxVrank1}.

\begin{remark}
{~}
\vspace*{-0.2cm}
\begin{itemize}
\item Computing $\prox^V_\f$ amounts to solving a scalar optimization problem that involves the computation of $\prox_{\f \circ D^{-1/2}}$. The latter can be much simpler to compute as $D$ is diagonal (beyond the obvious separable case that we will consider shortly). This is typically the case when $\f$ is the indicator of the $\ell_1$-ball or the canonical simple. The corresponding projector can be obtained in expected complexity $O(N\log N)$ by simple sorting the absolute values 
\item It is of course straightforward to compute $\prox^V_{\f^*}$ from $\prox^V_\f$ either using Theorem~\ref{theo:proxVrank1}, or using this theorem together with Corollary~\ref{cor:proxV} and the Sherman-Morrison inversion lemma.
\end{itemize}
\end{remark}

\subsubsection{Diagonal+rank-1: Separable case}
The following corollary is key to our novel optimization algorithm.
\begin{corollary}
\label{cor:proxVseprank1}
Assume that $\f \in \Gamma_0(\Hm)$ is separable, i.e. $\f(x)=\sum_{i=1}^N \f_i(x_i)$, and $V=D+uu^T$, where $D$ is diagonal with (strictly) positive diagonal elements $d_i$, and $u \in \RR^N$. Then,
\be
\label{eq:proxVsep}
\prox^V_{\f}(x) = \parenth{\prox_{\f_i/d_i}(x_i - v_i/d_i)}_i,
\ee
where $v = \alpha u$ and $\alpha$ is the unique root of
\be
\label{eq:roothsep}
\h(\alpha) = \pds{u}{x - \parenth{\prox_{\f_i/d_i}(x_i - \alpha u_i/d_i)}_i} + \alpha ~,
\ee
which is a Lipschitz continuous and strictly increasing function on $\RR$.
\end{corollary}

\begin{proof}
As $\f$ is separable and $D \in \sdp(N)$ is diagonal, applying Theorem~\ref{theo:proxVrank1} together with Lemma~\ref{lem:proxcalc}\ref{proxscale}-\ref{proxsep}, the desired result follows.
\end{proof}

\begin{proposition}
\label{rem:proxVseprank1}
Assume that for $1 \leq i \leq N$, $\prox_{{\f}_i}$ is piecewise affine on $\RR$ with $k_i \geq 1$ segments, \ie
\[
\prox_{\f_i}(x_i) = a_j x_i + b_j,  \quad t_j \leq x_i \leq t_{j+1}, j \in \{1,\ldots,k_i\} ~.
\] 
Let $k=\sum_{i=1}^N k_i$. Then $\prox^V_\f(x)$ can be obtained exactly by sorting at most the $k$ real values $\parenth{\tfrac{d_i}{u_i}(x_i - t_j)}_{(i,j) \in \{1,\ldots,N\} \times \{1,\ldots,k_i\}}$.
\end{proposition}

\begin{proof}
Recall that \eqref{eq:roothgen} has a unique solution. When $\prox_{\f_i}$ is piecewise affine with $k_i$ segments, it is easy to see that $\h(\alpha)$ in \eqref{eq:roothsep} is also piecewise affine with slopes and intercepts changing at the $k$ transition points $\parenth{\tfrac{d_i}{u_i}(x_i - t_j)}_{(i,j) \in \{1,\ldots,N\} \times \{1,\ldots,k_i\}}$. To get $\alpha^\star$, it is sufficient to isolate the unique segment that intersects the abscissa axis. This can be achieved by sorting the values of the transition points which can cost in average complexity $O(k\log k)$.
\end{proof}

\begin{remark}
\label{prop:proxVseprank1linear}
{~}
\vspace*{-0.2cm}
\begin{itemize}
\item The above computational cost can be reduced in many situations by exploiting e.g.~symmetry of the $\f_i's$, identical functions, \etc This turns out to be the case for many functions of interest, \eg $\ell_1$-norm, indicator of the $\ell_\infty$-ball or the positive orthant, and many others; see examples hereafter. 
\item Corollary~\ref{cor:proxVseprank1} can be extended to the ``block'' separable (i.e.~separable in subsets of coordinates) when $D$ is piecewise constant along the same block indices. 
\end{itemize}
\end{remark}

\subsubsection{Semi-smooth Newton method}
In many situations (see examples below), the root of $\h(\alpha)$ can be found exactly in polynomial complexity. If no closed-form is available, one can appeal to some efficient iterative method to solve \eqref{eq:roothgen} (or \eqref{eq:roothsep}). As $\h$ is Lipschitz-continuous, hence so-called Newton (slantly) differentiable, semi-smooth Newton are good such solvers, with the proviso that one can design a simple slanting function which can be algorithmically exploited.

The semi-smooth Newton method for the solution of \eqref{eq:roothgen} can be stated as the iteration
\be
\label{eq:semismoothnewton}
\alpha_{t+1} = \alpha_t - g(\alpha_t)^{-1}\h(\alpha_t) ~,
\ee
where $g$ is a generalized derivative of $\h$.

\begin{proposition}[Generalized derivative of $\h$]
\label{prop:generaldh}
If $\prox_{\f \circ D^{-1/2}}$ is Newton differentiable with generalized derivative $G$, then so is the mapping $\h$ with a generalized derivative 
\[
g(\alpha) = 1 + \pds{u}{D^{-1/2}\circ G(D^{1/2}x - \alpha D^{-1/2}u)\circ D^{-1/2} u} ~.
\]
Furthermore, $g$ is nonsingular with a uniformly bounded inverse on $\RR$.
\end{proposition}
\begin{proof}
This follows from linearity and the chain rule \cite[Lemma~3.5]{GriesseLorenz07}. The second statement follows strict increasing monotonicity of $\h$ as established in Theorem~\ref{theo:proxVrank1}.
\end{proof}

Thus, as $\h$ is Newton differentiable with nonsingular generalized derivative whose inverse is also bounded, the general semi-smooth Newton convergence theorem implies that \eqref{eq:semismoothnewton} converges super-linearly to the unique root of $\eqref{eq:roothgen}$.

\subsubsection{Examples} \label{sec:examples}
Many functions can be handled very efficiently using our results above. For instance, Table~\ref{tab:listOfFcn} summarizes a few of them where we can obtain either an exact answer by sorting when possible, or else by minimizing w.r.t. to a scalar variable (\ie finding the unique root of \eqref{eq:roothgen}).

\begin{table}
    \centering
    \begin{tabular}{ll}
        \toprule
        Function $\f$ & Algorithm \\
        \midrule
$\ell_1$-norm & Separable: exact in $O(N\log N)$ \\
Hinge & Separable: exact in $O(N\log N)$ \\
$\ell_\infty$-ball & Separable: exact in $O(N\log N)$ from $\ell_1$-norm by Moreau-identity \\
Box constraint &  Separable: exact in $O(N\log N)$ \\
Positivity constraint & Separable: exact in $O(N\log N)$ \\
$\ell_1$-ball & Nonseparable: semismooth Newton and $\prox_{\f \circ D^{-1/2}}$ costs $O(N\log N)$ \\
$\ell_\infty$-norm & Nonseparable: from projector on the $\ell_1$-ball by Moreau-identity \\
Canonical simplex & Nonseparable: semismooth Newton and $\prox_{\f \circ D^{-1/2}}$ costs $O(N\log N)$ \\
$\max$ function & Nonseparable: from projector on the simplex by Moreau-identity \\
        \bottomrule
    \end{tabular}
    \caption{Summary of functions which have efficiently computable rank-1 proximity operators}
    \label{tab:listOfFcn}
\end{table}

To put Proposition\ref{rem:proxVseprank1} on a more concrete footing, we briefly cover the positivity constraint explicitly.
Let $V=D+uu^T$ and $\f(x) = \indic_{\{x:\,x\ge 0\}}$. We will calculate
\be  \label{eq:nnlsEx}
\prox_{\f}^{V^{-1}}(x) = \argmin_{ y \ge 0} \frac{1}{2}\|y-x\|_{V^{-1}}^2
\ee
Since we work with $V^{-1}$ and not $V$, we will not use $\h(\alpha)$ but rather $\hh(\alpha)$ which will be used in a similar way to $\h$.  

If $(y,\lambda)$ is a primal-dual solution to \eqref{eq:nnlsEx}, the KKT conditions must be satisfied:
\be \label{eq:KKT}
 y \ge 0, \lambda \ge 0, \quad y^T\lambda = 0, \quad y = x + (D+uu^T)\lambda
\ee
Define the scalar $\alpha = u^T\lambda$. The key observation is that if $\alpha$ is known,
then the problem is solved since it becomes separable and the solution is
$$ y_i = \parenth{x_i + \alpha u_i}_+, \quad \lambda_i = \parenth{-(x_i + \alpha u_i)/d_i}_+,\, i=1,\ldots,N
$$
where $\parenth{x_i}_+ := \max(0,x_i)$.
Let $\lambda_i^{(\alpha)} := \parenth{-(x_i + \alpha u_i)/d_i}_+$, so we search
for a value of $\alpha$ such that $\alpha = u^T\lambda^{(\alpha)}$, or in other words,
a root of $\hh(\alpha) = \alpha -  u^T\lambda^{(\alpha)}$.

Define $\hat{\alpha}_i$ to be the sorted values of $(-x_i/u_i)$, so 
we see that $\hh$ is linear in the regions $[\hat{\alpha}_i,\hat{\alpha}_{i+1}]$
and so it is trivial to check if $\hh$ has a root in this region.
Thus the problem is reduced to finding the correct region $i$,
which can be done efficiently by a binary search over $\log_2(n)$ values
of $i$ since $\hh$ is monotonic. To see that $\hh$ is monotonic, we write
it as 
$$ 
\hh(\alpha) = \alpha + \sum_{i=1}^N \left( (u_ix_i + \alpha u_i^2 )/d_i \right) \chi_i(\alpha)
$$
where $\chi_i(\alpha)$ encodes the positivity constraint in the argument of $\parenth{\cdot}_+$ and is thus either $0$ or $1$, hence the slope is always positive.

\section{A primal rank 1 SR1 algorithm}
\label{sec:SR1}

Following the conventional quasi-Newton notation,
we let $B$ denote an approximation to the Hessian of $f$ and $H$ denote
an approximation to the inverse Hessian. 
All quasi-Newton methods update an approximation to the (inverse) Hessian
that satisfies the \emph{secant condition}:
\begin{equation}\label{eq:secant}
    H_k y_k = s_k, \quad y_k = \nabla f(x_k) - \nabla f(x_{k-1}), \quad s_k = x_k - x_{k-1}
\end{equation}

Algorithm~\ref{alg:main} follows the SR1 method~\cite{SR1}, which uses a rank-1 update
to the inverse Hessian approximation at every step.  The SR1 method
is perhaps less well-known than BFGS, but it has the crucial property
that updates are rank-1, rather than rank-2,
and it is described ``[SR1] has now taken its place alongside the BFGS method as the pre-eminent updating formula.''~\cite{GouldLectureNotes}.

We propose two important modifications to SR1. The first is to use limited-memory,
as is commonly done with BFGS. In particular, we use zero-memory, which means
that at every iteration, a new diagonal plus rank-one matrix is formed.
The other modification is to extend the SR1 method to the general setting
of minimizing $f + h$ where $f$ is smooth but $h$ need not be smooth;
this further generalizes the case when $h$ is an indicator function of a
convex set.  Every step of the algorithm replaces $f$ with a quadratic approximation, and keeps $h$ unchanged.
Because $h$ is left unchanged, the subgradient of $h$ is used in an \emph{implicit} manner,
in comparison to methods such as~\cite{YuVishwanathan10} that use an approximation
to $h$ as well and therefore take an \emph{explicit} subgradient step.

\begin{algorithm}[h]
    \caption{Sub-routine to compute the approximate inverse Hessian $H_k$ \label{alg:SR1} }
\begin{algorithmic}[1]
    \REQUIRE $k, s_k, y_k, \;  0< \gamma < 1, \; 0 < \tau_\text{min} <\tau_\text{max} $ 
\IF{$k=1$} 
    \STATE $H_0 \leftarrow \tau \Id_\Hm$ where $\tau > 0$ is arbitrary
    \STATE $u_k \leftarrow 0$
\ELSE
    \STATE   $ \BBa \leftarrow \frac{ \pds{s_k}{ y_k } }{ \norm{y_k}^2 } $
    \hfill \COMMENT{Barzilai-Borwein step length}
    \STATE  Project $\BBa$ onto $[\tau_\text{min},\tau_\text{max}]$
  \STATE $H_0 \leftarrow \gamma \BBa \Id_\Hm$
  \IF{ $\pds{s_k - H_0 y_k }{y_k} \le 10^{-8} \|y_k\|_2 \|s_k - H_0 y_k \|_2 $ }
        \STATE $u_k \leftarrow 0$  \hfill \COMMENT{Skip the quasi-Newton update}
  \ELSE
  \STATE $u_k \leftarrow (s_k - H_0 y_k)/\sqrt{\pds{s_k - H_0 y_k }{y_k}})$. 
  \ENDIF
 \ENDIF
 \RETURN $H_k = H_0 + u_ku_k^T$ \hfill \COMMENT{$B_k = H_k^{-1}$ can be computed via the Sherman-Morrison formula}
\end{algorithmic}  
\end{algorithm}

\newcommand{\gk}{\nabla f(x_k) }

\paragraph{Choosing $H_0$} \label{sec:h0}
In our experience, the choice of $H_0$ is best if scaled with a Barzilai-Borwein
spectral step length
\begin{equation}
\BBa=\pds{s_k}{y_k}/\pds{y_k}{y_k}
    \label{eq:BBa}
\end{equation}
(we call it $\BBa$ to distinguish it from the other Barzilai-Borwein step size
$\BBb = \pds{s_k}{s_k}/\pds{s_k}{y_k} \ge \BBa$).

In SR1 methods, the quantity $ \pds{s_k - H_0 y_k}{y_k}$ must be positive
in order to have a well-defined update for $u_k$. The update is:
\begin{equation}
    H_k = H_0 + u_k u_k^T,\quad u_k = (s_k - H_0 y_k )/\sqrt{ \pds{s_k - H_0 y_k}{y_k} }.
    \label{eq:Hk}
\end{equation}
For this reason, we choose $H_0 = \gamma \BBa \Id_\Hm$ with $0 < \gamma < 1$,
and thus $  0 \le \pds{s_k - H_0 y_k}{y_k} = (1-\gamma)\pds{s_k}{y_k}$.
If $\pds{s_k}{y_k}=0$,
then there is no symmetric rank-one update that satisfies the secant condition.
The inequality $\pds{s_k}{y_k} > 0$ is the \emph{curvature condition},
and it is guaranteed for all strictly convex objectives. 
Following the recommendation in~\cite{NocedalWright}, we skip updates
whenever $\pds{s_k}{y_k}$ cannot be guaranteed to be non-zero
given standard floating-point precision.

A value of $\gamma=0.8$ works well in most situations.
We have tested picking $\gamma$ adaptively,
as well as trying $H_0$ to be non-constant on the diagonal, but found no consistent improvements.

\section{Numerical experiments and comparisons}
\label{sec:results}

\newlength{\mySubfigSize}
\setlength{\mySubfigSize}{7.5cm}

\begin{figure}[h]
    \begin{subfigure}[b]{0.5\textwidth}
        \centering
        \includegraphics[width=\textwidth]{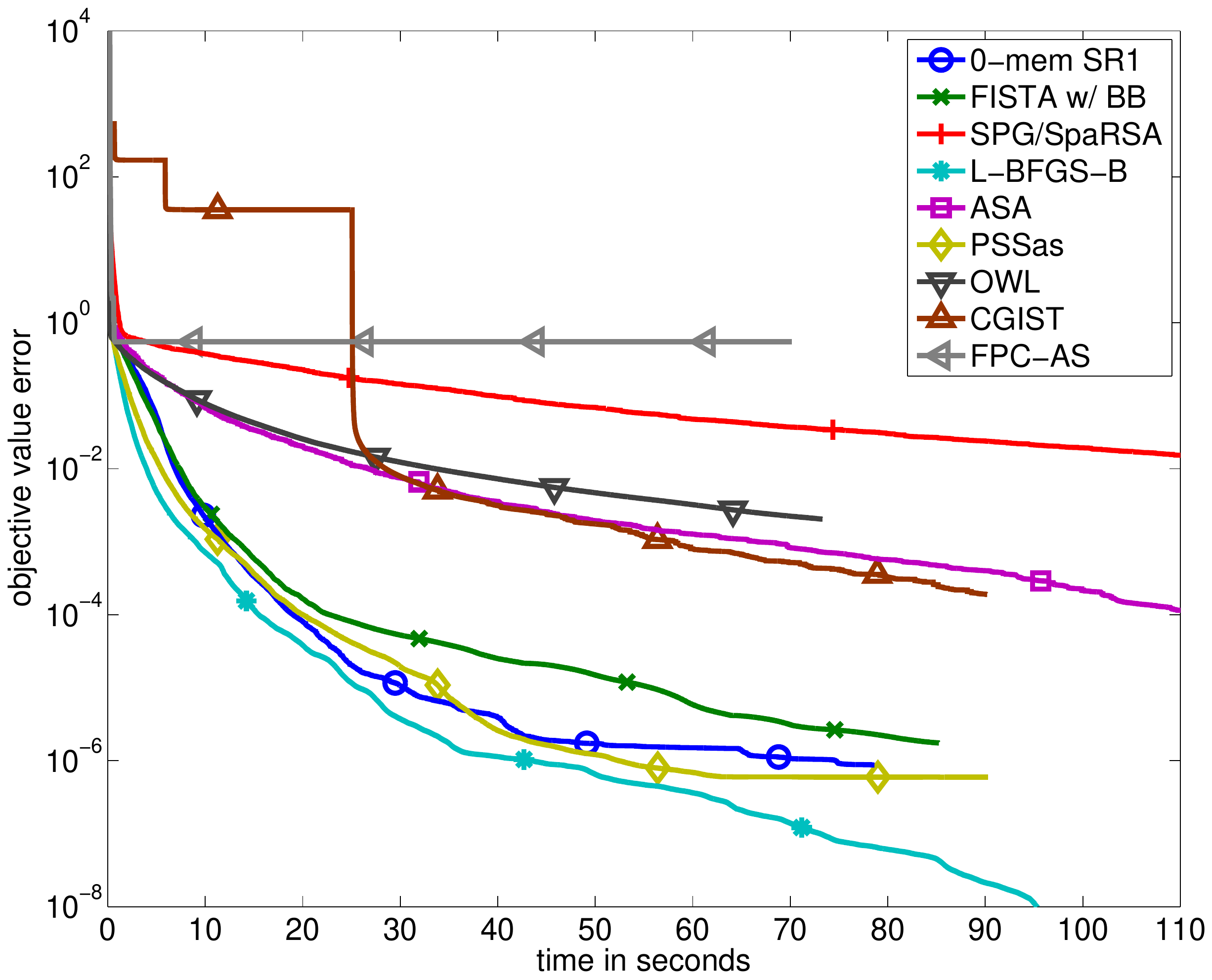}
        \caption{}
        \label{subfigA}
    \end{subfigure}%
    ~ 
    \begin{subfigure}[b]{0.5\textwidth}
        \centering
        \includegraphics[width=\textwidth]{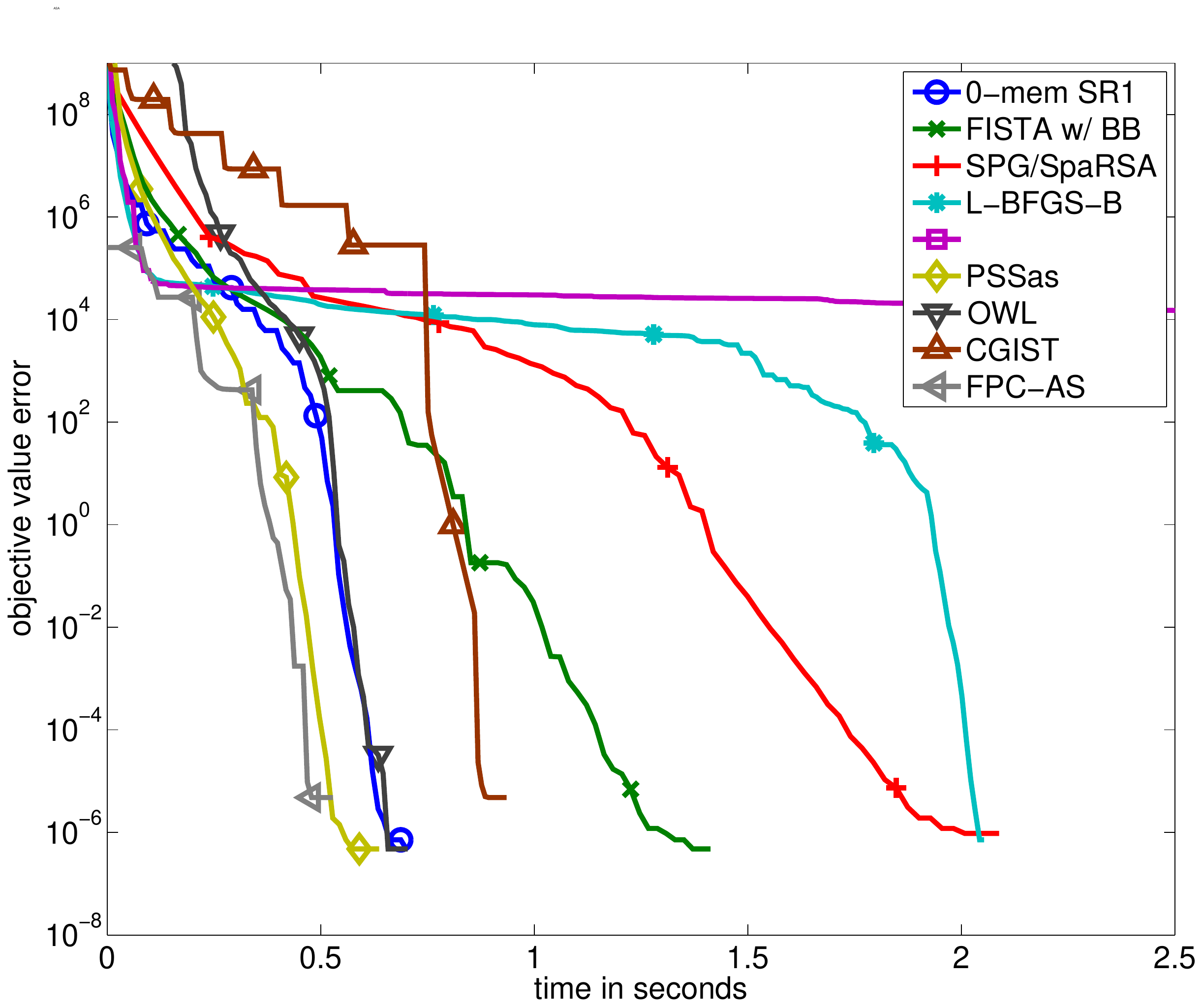}
        \caption{}
        \label{subfigB}
    \end{subfigure}
    \caption{(\subref{subfigA}) is first LASSO test, (\subref{subfigB})
    is second LASSO test}
    \label{fig:lasso}
\end{figure}

Consider the unconstrained LASSO problem \eqref{eq:LASSO}.
Many codes, such as \cite{DhillonQuasiNewton} and L-BFGS-B~\cite{LBFGSB}, handle only
non-negativity or box-constraints. Using the standard change of variables by introducing the positive and negative parts of $x$, the LASSO can be recast as
$$ \min_{ x_{+}, x_{-} \ge 0 } \frac{1}{2}\| Ax_+ - Ax_- -b\|^2 + \lambda \ones^T(x_+ + x_-)$$
and then $x$ is recovered via $x=x_+ - x_-$. With such a formulation solvers such as L-BFGS-B are applicable.
However, this constrained problem has twice the number of variables,
and the Hessian of the quadratic part changes from $A^TA$ to $\tilde{A}=\begin{pmatrix}A^TA & -A^TA \\-A^TA & A^TA \end{pmatrix}$ which
necessarily has (at least) $n$ degenerate 0 eigenvalues and adversely affects solvers.

A similar situation occurs with the hinge-loss function. Consider the shifted and reversed hinge loss function
$ h(x) = \max( 0, x )$. Then one can split $x=x_+ - x_-$, add constraints $x_+ \ge0, x_- \ge 0$,
and replace $h(x)$ with $ \ones^T(x_+)$. As before, the Hessian gains $n$ degenerate eigenvalues.

We compared our proposed algorithm on the LASSO problem. The first example, in Fig.~\ref{fig:lasso}\subref{subfigA}, is a typical example from compressed sensing that takes $A\in \RR^{m \times n}$ to have iid $\mathcal{N}(0,1)$ entries with $m=1500$ and $n=3000$. We set $\lambda=0.1$. L-BFGS-B does very well, followed closely by our proposed SR1 algorithm and PSSas. Note that L-BFGS-B and ASA are in Fortran and C, respectively (the other algorithms are in Matlab).
    
Our second example uses a square operator $A$ with dimensions $n = 13^3 = 2197$ chosen as a 3D discrete differential operator. This example stems from a numerical analysis problem to solve a discretized PDE as suggested by \cite{FletcherBB}. For this example, we set $\lambda=1$. 
For all the solvers, we use the same parameters as in the previous example.
Unlike the previous example, Fig.~\ref{fig:lasso}\subref{subfigB} now shows that L-BFGS-B is very slow on this problem. The FPC-AS method, very slow on the earlier test, is now the fastest. However, just as before, our SR1 method is nearly as good as the best algorithm.  This robustness is one benefit of our approach, since the method does not rely on active-set identifying parameters and inner iteration tolerances.

\section{Conclusions}
\label{sec:conclusion}

In this paper, we proposed a novel variable metric (quasi-Newton) forward-backward splitting algorithm, designed to efficiently solve non-smooth convex problems structured as the sum of a smooth term and a non-smooth one. We introduced a class of weighted norms induced by a diagonal+rank 1 symmetric positive definite matrices, and proposed a whole framework to compute a proximity operator in the weighted norm. The latter result is distinctly new and is of independent interest. We also provided clear evidence that the non-diagonal term provides significant acceleration over diagonal matrices.

The proposed method can be extended in several ways. Although we focused on forward-backward splitting, our approach can be easily extended to the new {\em generalized} forward-backward algorithm of \cite{raguet-gfb}. However, if we switch to a primal-dual setting, which is desirable because it can handle more complicated objective functionals, updating $B_k$ is non-obvious. Though one can think of non-diagonal pre-conditioning methods. 

Another improvement would be to derive efficient calculation for rank-2 proximity terms, thus allowing a 0-memory BFGS method. We are able to extend (result not presented here) Theorem~\ref{theo:proxVrank1} to diagonal+rank $r$ matrices. However, in general, one must solve an $r$-dimensional inner problem using the semismooth Newton method.

A final possible extension is to take $B_k$ to be diagonal plus rank-1 on diagonal blocks, since if $h$ is separable, this is still can be solved by our algorithm (see Remark~\ref{rem:proxVseprank1}). The challenge here is adapting this to a robust quasi-Newton update. For some matrices that are well-approximated by low-rank blocks, such as H-matrices~\cite{Hmatrices}, it may be possible to choose $B_k \equiv B$ to be a fixed preconditioner.

 \pdfbookmark[0]{Acknowledgments}{ack} 
 \subsubsection*{Acknowledgments}
{\small 
SB would like to acknowledge the Fondation Sciences Math\'{e}matiques de Paris for his fellowship.
}


\newpage
\appendix
\pdfbookmark[0]{Appendix}{appendix} 
\section{Elements from convex analysis}
\label{sec:appendix}
We here collect some results from convex analysis that are key for our proof. Some lemmata are listed without proof and can be either easily proved or found in standard references such as ~\cite{Rockafellar70,BauschkeCombettes11}. 

\subsection{Background}
\paragraph*{Functions}

\begin{definition}[Indicator function]
\label{def:ind}
Let $\C$ a nonempty subset of $\Hm$. The indicator function $\indic_{\C}$ of $\C$ is 
\[
\label{eq:ind}
\indic_{\C} (x) =
  \begin{cases}
    0, & \text{if } x \in \C ~ ,\\
    +\infty, & \text{otherwise}.
  \end{cases}
\]
$\dom(\indic_\C)=\C$.
\end{definition}

\begin{definition}[Infimal convolution]  
Let $\f_1$ and $\f_2$ two functions from $\Hm$ to $\RR \cup \acc{+\infty}$. Their infimal convolution is the function from $\Hm$ to $\RR \cup \acc{\pm\infty}$ defined by:
\[
    (\f_1 \infc \f_2)(x) = \inf\acc{\f_1(x_1) + \f_2(x_2): x_1+x_2=x} = \inf_{y\in\Hm} \f_1 (y) + \f_2(x-y)~.
\]
\end{definition}

\paragraph*{Conjugacy}

\begin{definition}[Conjugate]
Let $\f: \Hm \to \RR \cup \acc{+\infty}$ having a minorizing affine function. The conjugate or Legendre-Fenchel transform of $\f$ on $\Hm$ is the function $\f^*$ defined by
\[
\label{eq:conj}
\f^*(v) = \sup_{x \in \dom(\f)} \pds{v}{x} - \f(x) ~ .
\]
\end{definition}

\begin{lemma}[Calculus rules]
\label{lem:conjcalc}
{~}\\
\vspace*{-0.5cm}
\begin{enumerate}[label={\rm (\roman{*})}, ref={\rm (\roman{*})}]
\item \label{conjaddcst} $(\f(x)+t)^*(v) = \f^*(v)-t$.
\item \label{conjscale} $(t \f(x))^*(v) = tf^*(v/t)$, $t > 0$.
\item \label{conjlin} $(\f \circ A)^* = \f^*\circ\parenth{A^{-1}}^*$ if $A$ is a linear invertible operator.
\item \label{conjtrans} $(\f(x-x_0))^*(v) = \f^*(v) + \pds{v}{x_0}$.
\item \label{conjsep} Separability: $\parenth{\sum_{i=1}^n \f_i(x_i)}^*(v_1,\cdots,v_n) = \sum_{i=1}^n \f_i^*(v_i)$, where $(x_1,\cdots,x_n)\in\Hm_1\times\cdots\times\Hm_n$.
\item \label{conjsum} Conjugate of a sum: assume $\f_1,\f_2 \in \Gamma_0(\Hm)$ and the relative interiors of their domains have a nonempty intersection. Then
\[
(\f_1 + \f_2)^* = \f_1^* \infc \f_2^*~.
\]
\item \label{conjHV} Conjugate in $\Hm_V$ for $V \in \sdp(N)$: $\f^*_V(u) = \f^*(V u)$.
\end{enumerate}
\end{lemma}

\begin{lemma}[Conjugate of a degenerate quadratic function]
\label{lem:conjquad}
Let $Q$ be a symmetric positive semi-definite matrix. Let $Q^+$ be its Moore-Penrose pseudo-inverse. Then,
\be
\label{eq:conjquad}
\parenth{\frac{1}{2}\norm{y - \cdot}^2_Q}^*(v) = 
\begin{cases}
\frac{1}{2} \norm{y - v}_{Q^{+}}^2 & \text{if } v \in y + \range(Q) ~,\\
+\infty & \text{otherwise} ~.
\end{cases}
\ee
\end{lemma}

\begin{lemma}[Conjugate of a rank-1 quadratic function]
\label{lem:conjrank1}
Let $u \in \Hm$. Then,
\be
\label{eq}
\parenth{\frac{1}{2}\pds{u}{\cdot}^2}^*(v) = 
\begin{cases}
\frac{\norm{v}^2}{2\norm{u}^2} & \text{if } v \in \RR u ~,\\
+\infty & \text{otherwise} .
\end{cases}
\ee
\end{lemma}

\paragraph*{Subdifferential}
\begin{definition}[Subdifferential]
\label{def:subdiff}
The subdifferential of a proper convex function $\f \in \Gamma_0(\Hm)$ at $x \in \Hm$ is the set-valued map $\partial \f: \Hm \to 2^{\Hm}$
\[
\label{eq:subdiff1}
\partial \f(x) = \left\{v \in \Hm | \forall z \in \Hm, \f(z) \geq \f(x) + \pds{v}{z-x}\right\} ~.
\]
An element $v$ of $\partial \f$ is called a subgradient. 
\end{definition}
The subdifferential map $\partial \f$ is a maximal monotone operator from $\Hm \to 2^{\Hm}$.

\begin{lemma}
\label{lem:subdiffgrad}
If $\f$ is (G\^ateaux) differentiable at $x$, its only subgradient at $x$ is its gradient $\nabla \f(x)$. 
\end{lemma}

\begin{lemma}
Let $V \in \sdp(N)$. Then $V \partial \f$ is the subdifferential of $\f$ in $\Hm_V$ .
\end{lemma}

\paragraph*{Fenchel-Rockafellar duality}
The duality formula to be stated shortly will be very useful throughout the rest of the paper.
\begin{lemma}
\label{lem:fencheldual}
Let $\f \in \Gamma_0(\Hm)$ and $g \in \Gamma_0(\Km)$, and $A \eqdef L\cdot - y: \Hm \to \Km$ be a bounded affine operator, and $\Km=(\RR^m,\pds{\cdot}{\cdot})$. Suppose that $0 \in \ri\parenth{\dom g - A\parenth{\dom \f}}$. Then
\be
\label{eq:fencheldual1}
\inf_{x \in \Hm} \f(x) + g \circ A (x) = -\min_{u \in \Km} \f^*(-L^*u) + g^*(u) + \pds{u}{y} ~,
\ee
with the relashionships between $x^\star$ and $u^\star$, respectively the solutions of the primal and dual problems 
\be
\label{eq:fencheldual2}
\f(x^\star) + \f^*(-L^*u^\star) &=& \pds{-L^*u^\star}{x^\star} , \\
g(A x^\star) + g^*(u^\star) &=& \pds{u^\star}{A x^\star} ,
\ee
or equivalently
\be
\label{eq:fencheldual3}
x^\star \in \partial \f^*(-L^*u^\star) & \mathrm{and} & u^\star \in \partial g(A x^\star) ~,\\
-L^*u^\star \in \partial \f(x^\star) & \mathrm{and} & A x^\star \in \partial g^*(u^\star) ~.
\ee
\end{lemma}

\subsection{Proximal calculus in \texorpdfstring{$\Hm$}{H}}
\label{sec:appendix:prox}

\begin{definition}[Moreau envelope \cite{Moreau1962}]
\label{def:env} 
The function $\env{\f}{\rho}(x) = \inf_{z\in \Hm} \frac{1}{2\rho}\norm{x-z}^{2} + \f(z)$ for $0 < \rho < +\infty$ is the \textit{Moreau envelope} of index $\rho$ of $\f$.
\end{definition}

$\env{\f}{\rho}$ is also the infimal convolution of $\f$ with $\frac{1}{2\rho}\norm{\cdot}^2$.


\begin{lemma}
\label{lem:proxcalc}
{~}\\
\vspace{-0.5cm}
\begin{enumerate}[label={\rm (\roman{*})}, ref={\rm (\roman{*})}]
\item \label{proxtrans} Translation: $\prox_{\f(\cdot-y)}(x) = y + \prox_\f(x-y)$.
\item \label{proxscale} Scaling: $\forall \rho \in (-\infty,\infty), \prox_{\f(\rho \cdot)}(x) = \prox_{\rho^2f}(\rho x)/\rho$. 
\item \label{proxsep}   Separability~: let $(\f_i)_{1\le i\le n}$ a family of functions each in $\Gamma_0(\RR)$ and $\f(x) = \sum_{i=1}^N \f_i(x_i)$. Then $\f$ is in $\Gamma_0(\Hm)$ and $\prox_\f = \parenth{\prox_{\f_i}}_{1\le i \le N}$. 
\end{enumerate}
\end{lemma}

\begin{lemma}
\label{lem:envlip}
Let $\f \in \Gamma_0(\Hm)$. Then its Moreau envelope $\env{\f}{\rho}$ is convex and Fr\'echet-differentiable with $1/\rho$-Lipschitz gradient
\be
\nabla \env{\f}{\rho} = (\Id_\Hm - \prox_{\rho \f})/\rho.
\ee
\end{lemma}
\begin{lemma}[Moreau identity]
\label{lem:moreauident}
Let $\f \in \Gamma_0(\Hm)$, then for any $x \in \Hm$
\be
\prox_{\rho \f^*}(x) + \rho \prox_{\f/\rho}(x/\rho) = x, \forall ~ 0 < \rho < +\infty ~.
\ee
\end{lemma}
From Lemma~\ref{lem:moreauident}, we conclude that
\[
\prox_{\f^*} = \Id_\Hm - \prox_{\f}, \quad \prox_{\f^*}(x) \in \partial \f(x) ~.
\]

\section{Proofs}

\subsection{Proof of Lemma \ref{lem:moreauidentV}} \label{appendix:lem:moreauidentV}
\begin{proof}
We have
\begin{eqnarray*}
p = \prox^{V}_{\rho \f^*}(x) = (\Id_{\Hm_V} + V^{-1} \rho \f^*)^{-1}(x) 
& \iff & V(x-p) \in \partial (\rho \f^*)(p)\\
& \iff & p \in \partial \f(V(x - p)/\rho)\\
& \iff & Vx/\rho - (Vx - Vp)/\rho \in V\partial (\f/\rho)(V(x - p)/\rho)\\
& \iff & V(x - p)/\rho = (\Id_{\Hm_V}+V \partial (\f/\rho))^{-1}(Vx) \\
& \iff & x = p + \rho V^{-1} \circ (\Id_{\Hm_V}+V \partial (\f/\rho))^{-1}(Vx) ~.
\end{eqnarray*}
\end{proof}

\subsection{Proof of Theorem~\ref{theo:proxVrank1}} \label{appendix:thm:proxVrank1}
\begin{proof}
Let $p=\prox^V_{\f}(x)$. Then, we have to solve
\begin{eqnarray}
& &       \min_{z} \frac{1}{2} \norm{x - z}_V^2 + \f(z) \nonumber\\
& \iff &  \min_{z} \parenth{\frac{1}{2}\norm{z}_D^2 - \pds{x}{z}_D + \f(z)} + \pds{x - z}{uu^T(x - z)} \nonumber\\
y = D^{1/2} z, q=D^{-1/2}u & \iff & \min_{y} \parenth{\frac{1}{2}\norm{y}^2 - \pds{D^{1/2}x}{y} + \f\circ D^{-1/2}(y)} + \pds{D^{1/2}x - y}{qq^T(D^{1/2}x - y)} \nonumber\\
(\text{\small Lemma~\ref{lem:fencheldual}\eqref{eq:fencheldual1}}) & \iff &  \min_{v} \parenth{\frac{1}{2}\norm{\cdot}^2 - \pds{D^{1/2}x}{\cdot} + \f\circ D^{-1/2}}^*(-v) + \parenth{\pds{D^{1/2}x - \cdot}{qq^T(D^{1/2}x - \cdot)}}^*(v) \nonumber\\
\hspace*{-0.5cm}(\text{\small Lemma~\ref{lem:conjrank1} and Lemma~\ref{lem:conjcalc}\ref{conjtrans}}) & \iff &  \min_{v \in \RR q} \parenth{\frac{1}{2}\norm{\cdot}^2 - \pds{D^{1/2}x}{\cdot} + \f\circ D^{-1/2}}^*(-v) + \frac{\norm{v}^2}{2\norm{q}^2} + \pds{D^{1/2}x}{v} \nonumber\\
(\text{\small Lemma~\ref{lem:conjcalc}\ref{conjsum}-\ref{conjlin}}) & \iff &  \min_{v \in \RR q} \parenth{\parenth{\frac{1}{2}\norm{\cdot}^2 - \pds{D^{1/2}x}{\cdot}}^* \infc (\f^* \circ D^{1/2})}(-v) + \frac{\norm{v}^2}{2\norm{q}^2} + \pds{D^{1/2}x}{v} \nonumber\\
																 & \iff &  \min_{v \in \RR q} \parenth{\parenth{\frac{1}{2}\norm{D^{1/2}x+\cdot}^2}  \infc (\f^* \circ D^{1/2})}(-v_i) + \frac{\norm{v}^2}{2\norm{q}^2} + \pds{D^{1/2}x}{v} \nonumber\\
(\text{\small Definition~\ref{def:env}})		 & \iff &  \min_{v \in \RR q} \env{{\parenth{\f^*\circ D^{1/2}}}}{1}(D^{1/2}x - v) + \frac{\norm{v}^2}{2\norm{q}^2} + \pds{D^{1/2}x}{v} ~.
\label{eq:mindual}
\end{eqnarray}
By virtue of Lemma~\ref{lem:envlip}, $\env{{\parenth{\f^*\circ D^{1/2}}}}{{1}}$ is continuously differentiable with 1-Lipschitz gradient. Together with Lemma~\ref{lem:fencheldual}\eqref{eq:fencheldual3}, \ref{lem:subdiffgrad} and \ref{lem:moreauident}, this yields
\begin{eqnarray*}
p = D^{-1/2} \circ \nabla \env{{\parenth{\f^*\circ D^{1/2}}}}{{1}}(D^{1/2}x - v^\star)
&=& D^{-1/2} \circ \parenth{\Id_\Hm - \prox_{\f^*\circ D^{1/2}}}(D^{1/2}x - v^\star)\\
&=& D^{-1/2} \circ \prox_{\f\circ D^{-1/2}}\circ D^{1/2}(x - D^{-1/2}v^\star),
\end{eqnarray*}
where $v^\star$ is the unique solution to the above dual problem \eqref{eq:mindual}. This problem amounts to minimizing a proper convex smooth continuously differentiable objective with a Lipschitz gradient over a linear set. The latter can be parametrized by a real scalar $\alpha$ such that $v=\alpha q = \alpha D^{-1/2} u$, and is then equivalent to solving the scalar strongly convex smooth optimization problem
\begin{eqnarray}
\label{eq:minalpha}
\min_{\alpha \in \RR} \env{{\parenth{\f^*\circ D^{1/2}}}}{{1}}(D^{1/2}x - \alpha D^{-1/2}u) + \frac{\alpha^2}{2} + \alpha \pds{x}{u} ~,
\end{eqnarray}
whose solution $\alpha^\star$ is unique. This is equivalent to saying that $\alpha^\star$ is the unique root of
\be
\label{eq:halpha}
\h(\alpha) \eqdef \pds{u}{x - D^{-1/2} \circ \prox_{\f \circ D^{-1/2}}\circ D^{1/2}(x - \alpha D^{-1} u)} + \alpha ~,
\ee
where we used again Lemma~\ref{lem:envlip} and \ref{lem:moreauident}. Lipschitz continuity of $\h(\alpha)$ follows from non-expansiveness of the proximal mapping, and the Lipschitz constant is straightforward from the triangle and Cauchy-Schwartz inequalities.

Let's turn now to strict increasing monotonicity of $\h$. Let $\beta > \alpha$. Denote the operator $P = \prox_{\f \circ D^{-1/2}}$. Then,
\begin{eqnarray*}
\h(\beta)-\h(\alpha) 
&=& (\beta-\alpha)  - \pds{D^{-1/2} u}{P\circ D^{1/2}(x - \beta D^{-1} u)-P\circ D^{1/2}(x - \alpha D^{-1} u)} \\
&=& (\beta-\alpha) + (\beta-\alpha)^{-1} \pds{-(\beta-\alpha) D^{-1/2} u}{P(D^{1/2}x - \beta D^{-1/2} u)-P(D^{1/2}x - \alpha D^{-1/2} u)} \\
&\geq& (\beta-\alpha) + (\beta-\alpha)^{-1}\norm{P(D^{1/2}x - \beta D^{-1/2} u)-P(D^{1/2}x - \alpha D^{-1/2} u)}^2 \\
&>& 0 ~,
\end{eqnarray*}
where the first inequality is a consequence of the fact that the proximal mapping is firmly non-expansive.
\end{proof}

\pdfbookmark[0]{References}{references} 
\bibliographystyle{unsrt}

\small{

\begin{thebibliography}{10}

\bibitem{BauschkeCombettes11}
H.~H. Bauschke and P.~L. Combettes.
\newblock {\em Convex Analysis and Monotone Operator Theory in {H}ilbert
  Spaces.}
\newblock Springer-Verlag, New York, 2011.

\bibitem{LBFGSB}
R.~H. Byrd, P.~Lu, J.~Nocedal, and C.~Zhu.
\newblock A limited memory algorithm for bound constrained optimization.
\newblock {\em SIAM J. Sci. Computing}, 16(5):1190--1208, 1995.

\bibitem{CombettesPesquetChapter}
P.~L. Combettes and J.~C. Pesquet.
\newblock Proximal splitting methods in signal processing.
\newblock In H.~H. Bauschke, R.~S. Burachik, P.~L. Combettes, V.~Elser, D.~R.
  Luke, and H.~Wolkowicz, editors, {\em Fixed-Point Algorithms for Inverse
  Problems in Science and Engineering}, pages 185--212. Springer-Verlag, New
  York, 2011.

\bibitem{SPG}
E.~G. Birgin, J.~M. Mart\'inez, and M.~Raydan.
\newblock Nonmonotone spectral projected gradient methods on convex sets.
\newblock {\em SIAM J. Optim.}, 10(4):1196--1211, 2000.

\bibitem{WrightSparsa08}
S.~Wright, R.~Nowak, and M.~Figueiredo.
\newblock Sparse reconstruction by separable approximation.
\newblock {\em {IEEE} Transactions on Signal Processing}, 57, 2009.
\newblock 2479--2493.

\bibitem{BB88}
J.~Barzilai and J.~Borwein.
\newblock Two point step size gradient method.
\newblock {\em IMA J. Numer. Anal.}, 8:141--148, 1988.

\bibitem{FISTA}
A.~Beck and M.~Teboulle.
\newblock A fast iterative shrinkage-thresholding algorithm for linear inverse
  problems.
\newblock {\em SIAM J. on Imaging Sci.}, 2(1):183--202, 2009.

\bibitem{restart}
B.~O'Donoghue and E.~Cand\`es.
\newblock Adaptive restart for accelerated gradient schemes.
\newblock {\em Preprint: arXiv:1204.3982}, 2012.

\bibitem{ChambollePock11b}
T.~Pock and A.~Chambolle.
\newblock Diagonal preconditioning for first order primal-dual algorithms in
  convex optimization.
\newblock In {\em ICCV}, 2011.

\bibitem{ChenRockafellar97}
G.~H.-G. Chen and R.~T. Rockafellar.
\newblock Convergence rates in forward--backward splitting.
\newblock {\em SIAM Journal on Optimization}, 7(2):421--444, 1997.

\bibitem{L-BFGS-B-97}
C.~Zhu, R.~H. Byrd, P.~Lu, and J.~Nocedal.
\newblock Algorithm 778: {L-BFGS-B}: {F}ortran subroutines for large-scale
  bound-constrained optimization.
\newblock {\em ACM Trans. Math. Software}, 23(4):550--560, 1997.

\bibitem{LBFGSB2011}
Jos{\'e}~Luis Morales and Jorge Nocedal.
\newblock Remark on \"algorithm 778: {L-BFGS-B}: {F}ortran subroutines for
  large-scale bound constrained optimization\".
\newblock {\em ACM Trans. Math. Softw.}, 38(1):7:1--7:4, 2011.

\bibitem{HagerZhang06}
W.~W. Hager and H.~Zhang.
\newblock A new active set algorithm for box constrained optimization.
\newblock {\em SIAM J. Optim.}, 17:526--557, 2006.

\bibitem{AndrewGao07}
A.~Andrew and J.~Gao.
\newblock Scalable training of $l_1$-regularized log-linear models.
\newblock In {\em ICML}, 2007.

\bibitem{Schmidt2007a}
M.~Schmidt, G.~Fung, and R.~Rosales.
\newblock Fast optimization methods for l1 regularization: A comparative study
  and two new approaches.
\newblock In {\em European Conference on Machine Learning}, 2007.

\bibitem{FPCAS}
Z.~Wen, W.~Yin, D.~Goldfarb, and Y.~Zhang.
\newblock A fast algorithm for sparse reconstruction based on shrinkage,
  subspace optimization and continuation.
\newblock {\em SIAM J. Sci. Comput.}, 32(4):1832--1857, 2010.

\bibitem{GoldsteinSetzer11}
T.~Goldstein and S.~Setzer.
\newblock High-order methods for basis pursuit.
\newblock Technical report, CAM-UCLA, 2011.

\bibitem{YuVishwanathan10}
J.~Yu, S.V.N. Vishwanathan, S.~Guenter, and N.~Schraudolph.
\newblock A quasi-{N}ewton approach to nonsmooth convex optimization problems
  in machine learning.
\newblock {\em J. Machine Learning Research}, 11:1145--1200, 2010.

\bibitem{projQuasiNewton09}
M.~Schmidt, E.~van~den Berg, M.~Friedlander, and K.~Murphy.
\newblock Optimizing costly functions with simple constraints: A limited-memory
  projected quasi-{N}ewton algorithm.
\newblock In {\em AISTATS}, 2009.

\bibitem{projNewton11}
M.~Schmidt, D.~Kim, and S.~Sra.
\newblock Projected {N}ewton-type methods in machine learning.
\newblock In S.~Sra, S.~Nowozin, and S.Wright, editors, {\em Optimization for
  Machine Learning}. MIT Press, 2011.

\bibitem{Saunders12}
J.~D. Lee, Y.~Sun, and M.~A. Saunders.
\newblock Proximal {N}ewton-type methods for minimizing convex objective
  functions in composite form.
\newblock {\em Preprint: arXiv:1206.1623}, 2012.

\bibitem{Moreau1962}
J.-J. Moreau.
\newblock {Fonctions convexes duales et points proximaux dans un espace
  hilbertien}.
\newblock {\em CRAS S\'er. A Math.}, 255:2897--2899, 1962.

\bibitem{GriesseLorenz07}
R.~Griesse and D.~A. Lorenz.
\newblock A semismooth {N}ewton method for {T}ikhonov functionals with sparsity
  constraints.
\newblock {\em Inverse Problems}, 24(3):035007, 2008.

\bibitem{SR1}
C.~Broyden.
\newblock Quasi-{N}ewton methods and their application to function
  minimization.
\newblock {\em Math. Comp.}, 21:577--593, 1967.

\bibitem{GouldLectureNotes}
N.~Gould.
\newblock Seminal papers in nonlinear optimization.
\newblock In {\em An introduction to algorithms for continuous optimization}.
  Oxford University Computing Laboratory, 2006.
\newblock
  \url{http://www.numerical.rl.ac.uk/nimg/course/lectures/paper/paper.pdf}.

\bibitem{NocedalWright}
J.~Nocedal and S.~Wright.
\newblock {\em Numerical Optimization}.
\newblock Springer, 2nd edition, 2006.

\bibitem{DhillonQuasiNewton}
I.~Dhillon, D.~Kim, and S.~Sra.
\newblock Tackling box-constrained optimization via a new projected
  quasi-{N}ewton approach.
\newblock {\em SIAM J. Sci. Comput.}, 32(6):3548--3563, 2010.

\bibitem{FletcherBB}
Roger Fletcher.
\newblock On the {B}arzilai-{B}orwein method.
\newblock In Liqun Qi, Koklay Teo, Xiaoqi Yang, Panos~M. Pardalos, and
  Donald~W. Hearn, editors, {\em Optimization and Control with Applications},
  volume~96 of {\em Applied Optimization}, pages 235--256. Springer US, 2005.

\bibitem{raguet-gfb}
H.~Raguet, J.~Fadili, and G.~Peyr{\'e}.
\newblock Generalized forward-backward splitting.
\newblock Technical report, Preprint Hal-00613637, 2011.

\bibitem{Hmatrices}
W.~Hackbusch.
\newblock A sparse matrix arithmetic based on {H}-matrices. {P}art {I}:
  Introduction to {H}-matrices.
\newblock {\em Computing}, 62:89--108, 1999.

\bibitem{Rockafellar70}
R.~T. Rockafellar.
\newblock {\em Convex Analysis}.
\newblock Princeton University Press, 1970.

\end{thebibliography}

}

\end{document}